\documentclass[12pt]{amsart}
\usepackage{amssymb}
\usepackage{amsfonts}
\usepackage{latexsym}
\usepackage{amscd}
\usepackage[mathscr]{euscript}
\usepackage{xy} \xyoption{all}
\usepackage[active]{srcltx}

\usepackage{tkz-graph}
\usepackage{tikz-cd}
\usetikzlibrary{matrix,arrows,calc,automata,trees}

\usepackage{tikz}

\usepackage{color,graphics}

\usepackage[shortlabels]{enumitem}
\usepackage{url}

\newcommand{\Z}{{\mathbb{Z}}}

\newcommand{\C}{{\mathbb{C}}}

\newcommand{\ol}{\overline}
\newcommand{\uloopr}[1]{\ar@'{@+{[0,0]+(-4,5)}@+{[0,0]+(0,10)}@+{[0,0] +(4,5)}}^{#1}}
\newcommand{\uloopd}[1]{\ar@'{@+{[0,0]+(5,4)}@+{[0,0]+(10,0)}@+{[0,0]+ (5,-4)}}^{#1}}
\newcommand{\dloopr}[1]{\ar@'{@+{[0,0]+(-4,-5)}@+{[0,0]+(0,-10)}@+{[0, 0]+(4,-5)}}_{#1}}
\newcommand{\dloopd}[1]{\ar@'{@+{[0,0]+(-5,4)}@+{[0,0]+(-10,0)}@+{[0,0 ]+(-5,-4)}}_{#1}}

\newcommand{\calA}{{\mathcal A}}
\newcommand{\calB}{{\mathcal B}}

\newcommand{\calM}{{\mathcal M}}

\newcommand{\calQ}{{\mathcal Q}}

\newcommand{\luloop}[1]{\ar@'{@+{[0,0]+(-8,2)}@+{[0,0]+(-10,10)}@+{[0, 0]+(2,2)}}^{#1}}

\newcommand{\dotedge}{\ar@{.}}
\newcommand{\eqedge}{\ar@{=}}

\newcommand{\LPeqRP}{{\mathrm{LP}} \overset{*}{\sim} {\mathrm{RP}}}

\numberwithin{equation}{section}

\theoremstyle{plain}
\newtheorem{theorem}{Theorem}[section]
\newtheorem*{theorem*}{Theorem}
\newtheorem{lemma}[theorem]{Lemma}

\newtheorem{corollary}[theorem]{Corollary}

\theoremstyle{definition}
\newtheorem{definition}[theorem]{Definition}

\newtheorem*{remark*}{Remark}

\newtheorem*{assumption*}{Assumption}

\newtheorem{notation}[theorem]{Notation}

\begin{document}

\null\vskip-1cm

\title[von Neumann factor]{Uniqueness of the von Neumann continuous factor}%
\author{Pere Ara}
\address{Departament de Matem\`atiques, Universitat Aut\`onoma de Barcelona,
08193 Bellaterra (Barcelona), Spain.} \email{para@mat.uab.cat}
\author{Joan Claramunt}
\address{Departament de Matem\`atiques, Universitat Aut\`onoma de Barcelona,
08193 Bellaterra (Barcelona), Spain.} \email{jclaramunt@mat.uab.cat}
\thanks{Both authors were partially supported by DGI-MINECO-FEDER through the grants
MTM2014-53644-P and MTM2017-83487-P} 
\date{\today}

\maketitle

\begin{abstract}
For a division ring $D$, denote by $\mathcal M_D$ the $D$-ring obtained as the completion of the direct limit $\varinjlim_n M_{2^n}(D)$
with respect to the metric induced by its unique rank function. We prove that, for any ultramatricial $D$-ring $\mathcal B$ and any
non-discrete extremal pseudo-rank function $N$ on $\mathcal B$, there is an isomorphism of $D$-rings $\ol{\mathcal B} \cong \mathcal M_D$, where $\ol{\mathcal B}$ stands
for the completion of $\mathcal B$ with respect to the pseudo-metric induced by $N$.
This generalizes a result of von Neumann. We also show a corresponding uniqueness result for $*$-algebras over fields $F$ with positive definite involution, where the
algebra $\mathcal M_F$ is endowed with its natural involution coming from the $*$-transpose involution on each of the factors $M_{2^n}(F)$.
\end{abstract}


\section{Introduction}
\label{sect:intro}

Murray and von Neumann showed in \cite[Theorem XII]{MvN} a uniqueness result for approximately finite von Neuman algebra factors of type II$_1$.
This unique factor $\mathcal R$ is called the {\it hyperfinite II$_1$-factor} and plays a very important role in the theory of von Neumann algebras.
It was shown later by Alain Connes \cite{connes} that the factor $\mathcal R$ is characterized (among II$_1$-factors) by various other properties, such as injectivity
(in the operator space sense), semidiscreteness or Property P. It is in particular known (e.g. \cite[Theorem 3.8.2]{SS}) that, for an infinite countable discrete group $G$
whose non-trivial conjugacy classes are all infinite, the group von Neumann algebra $\mathcal N (G)$ is isomorphic to $\mathcal R$ if and only if $G$ is an amenable group.
(The groups with the above property on the conjugacy classes are termed ICC-groups.)

Von Neumann also considered a purely algebraic analogue of the above situation, as follows.
For a field $K$, the direct limit $\varinjlim_{n} M_{2^n}(K)$ with respect to the block diagonal embeddings $x\mapsto \begin{pmatrix} x & 0 \\ 0 & x  \end{pmatrix}$ is a (von Neumann) regular ring,
which admits a unique rank function (see below for the definition of rank function). The completion of $\varinjlim_{n} M_{2^n}(K)$, denoted here by $\mathcal M_K$, with respect to the induced rank metric,
is a complete regular ring with a unique rank function,
which is a {\it continuous factor}, i.e., it is a right and left self-injective ring and the set of values of the rank function fills the unit interval $[0,1]$.
There are recent evidences \cite{elekTRANS, elek, elek2} that the factor $\mathcal M _K$ could play a role in algebra which is similar to the role played by the unique hyperfinite factor $\mathcal R$
in the theory of operator algebras.
In particular, Elek has shown in \cite{elek} that, if $\Gamma = \Z_2\wr \Z$ is the lamplighter group, then the continuous factor $c(\Gamma )$ obtained by taking the rank completion of
the $*$-regular closure of $\C [\Gamma ]$ in the $*$-algebra $\mathcal U (\Gamma )$ of unbounded operators affiliated to $\mathcal N (\Gamma )$, is isomorphic to $\mathcal M_{\C}$.

This raises the question of what uniqueness properties the von Neumann factor $\mathcal M_K$ has, and whether we can formulate similar characterizations to those in the
seminal paper by Connes \cite{connes}. As von Neumann had already shown (and was published later by Halperin \cite{Halp}), $\mathcal M_K$ is isomorphic
to the factor obtained from any {\it factor sequence} $(p_i)$, that is,
$$\mathcal M_K \cong \ol{\varinjlim M_{p_i}(K)} ,$$
where $(p_i)$ is a sequence of positive integers converging to $\infty$ and such that $p_i$ divides $p_{i+1}$ for all $i$. Here the completion is taken with respect to the unique rank function on the
direct limit.

The purpose of this paper is to obtain stronger uniqueness properties of the factor $\mathcal M_K$. Specifically, we show that if
$\calB$ is an ultramatricial $K$-algebra and $N$ is a non-discrete extremal pseudo-rank function on $\calB$, then the completion of $\calB$ with respect to $N$ is
necessarily isomorphic to $\mathcal M_K$.  We also derive a characterization of the factor $\mathcal M_K$ by a local approximation property (see Theorem \ref{thm:local-ultra}).
This will be used in \cite{AC} to generalize Elek's result to arbitrary fields $K$ of characteristic $\ne 2$, using a concrete approximation of the group algebra $K[\Gamma]$ by matricial algebras.
It is also worth to mention that, as a consequence of our result and \cite[Theorem 2.8]{Good-centers}, one obtains that the center of an algebra $\calQ$ satisfying properties (ii) or (iii) in Theorem
\ref{thm:local-ultra} is the base field $K$.

Gabor Elek and Andrei Jaikin-Zapirain have recently raised the question of whether, for any subfield $k$ of $\C$ closed under complex conjugation, and any countable amenable ICC-group $G$, the rank completion
$c(k[G])$ of the $*$-regular closure of $k[G]$ in $\mathcal U (G)$ is either of the form $M_n(D)$ or of the form $\mathcal M _D:= D\otimes_k \mathcal M_k$, where $D$ is a division ring with center $k$.
In view of this question, it is natural to obtain uniqueness results in the slightly more general setting of $D$-rings over a division ring $D$, and also in the setting of rings with involution, since
in the above situation, the algebras have a natural involution which is essential even to define the corresponding completions. We address these questions in the final two sections.

\section{von Neumann's continuous factor}
\label{sec:generalvN}

A ring $R$ is said to be {\it (von Neumann) regular} in case for each $x\in R$ there exists $y\in R$ such that $x= xyx$. We refer the reader to \cite{vnrr}
for the general theory of regular rings.

We recall the definition of a pseudo-rank function on a general unital ring.

\begin{definition}
 \label{def:rankring} A \emph{pseudo-rank function} on a unital ring $R$ is a function $N \colon R\to [0,1]$ satisfying the following properties:
 \begin{enumerate}
 \item $N(1) = 1$.
 \item $N (a+b)\le N (a) + N (b)$ for all $a,b\in R$.
  \item $N (ab)\le N (a), N (b)$ for all $a,b\in R$.
  \item If $e,f \in R$ are orthogonal idempotents, then $N(e+f)=N (e) +N (f) $.
 \end{enumerate}
A \emph{rank function} on $R$ is a pseudo-rank function $N$ such that $N(x)= 0$ implies $x=0$.
\end{definition}

Any pseudo-rank function $N$ on a ring $R$ induces a pseudo-metric by $d(x,y)= N(x-y)$ for $x,y\in R$. If, in addition, $R$ is regular, then the completion
of $R$ with respect to this pseudo-metric is again a regular ring $\ol{R}$, and $\ol{R}$ is complete with respect to the unique extension $\ol{N}$ of $N$ to
a rank function on $\ol{R}$ (\cite[Theorem 19.6]{vnrr}). The space of pseudo-rank functions $\mathbb P (R)$ on a regular ring $R$ is a Choquet simplex (\cite[Theorem 17.5]{vnrr}), and the
completion $\ol{R}$ of $R$ with respect to $N\in \mathbb P (R)$ is
a simple ring if and only if $N$ is an extreme point in $\mathbb P (R)$ (\cite[Theorem 19.14]{vnrr}).

For a field $K$, a {\it matricial} $K$-algebra is a $K$-algebra which is isomorphic to an algebra of the form
$$M_{n(1)}(K) \times M_{n(2)}(K) \times \cdots \times M_{n(k)}(K)$$
for some positive integers $n(1),n(2),\dots , n(k)$. An {\it ultramatricial} $K$-algebra is an algebra which is isomorphic to a
direct limit $\varinjlim _n \mathcal A_n$ of a sequence of matricial $K$-algebras $\mathcal A_n$ and unital algebra homomorphisms
$\varphi_n\colon \mathcal A_n\to \mathcal A_{n+1}$, see \cite[Chapter 15]{vnrr}.

Let $K$ be a field. Write $\calM= \mathcal M_K$ for the rank completion of the  direct limit $\varinjlim _{n} M_{2^n}(K)$ with respect to its unique rank function.
Von Neumann proved a uniqueness property for $\mathcal M$. We are going to extend it to ultramatricial algebras. The proof follows the steps in the paper
by Halperin \cite{Halp} (based on von Neumann's proof), but the proof is considerably more involved. Indeed, we will obtain a uniqueness result for the class of continuous factors
which have a local matricial structure.

By a {\it continuous factor} we understand a simple, regular, (right and left) self-injective ring $\calQ$ of type $II_f$ (see \cite[Chapter 10]{vnrr} for the definition of the types and
for the structure theory of regular self-injective rings).
It follows from \cite[Corollary 21.14]{vnrr}
that $\calQ$ admits a unique rank function, denoted here by $N_{\calQ}$, and that $\calQ$ is complete in the $N_{\calQ}$-metric.
Also, it follows easily from the structure theory of regular self-injective rings (\cite[Chapter 10]{vnrr}) that
$N_{\calQ} (\calQ )= [0,1]$.

The adjective ``continuous'' used here refers to the fact that $N_{\calQ}$ takes a ``continuous'' set of values, in contrast with the algebra of
finite matrices, where the rank function takes only a finite number of values. Note however that any regular self-injective ring $R$ is a right (left) continuous regular ring,
in the technical sense that the lattice of principal right (left) ideals is continuous, see \cite{vnrr}. The latter property will play no explicit role in the present paper.

 We will show the following result:

 \begin{theorem}
 \label{thm:local-ultra}
Let $\calQ$ be a continuous factor, and assume there exists a dense subalgebra (with respect to the $N_{\calQ}$-metric topology) $\calQ_0 \subseteq \calQ$
of countable $K$-dimension. The following are equivalent:
\begin{enumerate}[(i)]
 \item $\calQ \cong \calM_K$.
 \item $\calQ \cong \overline{\calB}$ for a certain ultramatricial $K$-algebra $\calB$, where the completion of $\calB$ is taken with respect to the metric induced by
 an extremal pseudo-rank function on $\calB$.
 \item For every $\varepsilon > 0$ and $x_1,...,x_n \in \calQ$, there exists a matricial $K$-subalgebra $\calA$ of $\calQ$ and elements $y_1,...,y_n \in \calA$ such that
$$N_{\mathcal Q} (x_i-y_i) < \varepsilon \qquad (i=1,\dots , k ) .$$
\end{enumerate}
\end{theorem}

(i)$\implies $(ii)$\implies $(iii) is clear. For the proof of the
implication (iii)$\implies $(i) we will use a method similar to the one used in \cite{Halp}. However the technical complications are much higher here.

We first prove a lemma, and show the implication (iii)$\implies $(i) assuming that the hypotheses of the lemma are satisfied.
After this is done, we will show how to construct (using (iii)) the sequences, algebras, and homomorphisms appearing in this lemma.

Given a factor sequence $(p_i)$, the natural block-diagonal unital embeddings $M_{p_i}(K) \to M_{p_{i+1}}(K)$ will be denoted by $\gamma_{i+1,i}$.
If $j>i$, the map $\gamma_{j,i}\colon M_{p_i}(K)\to M_{p_j}(K)$ will denote the composition $\gamma_{j,i}= \gamma_{j,j-1}\circ \cdots \circ \gamma_{i+1,i}$,
and the map $\gamma_{\infty, i}\colon M_{p_i}(K) \to \varinjlim _n M_{p_n}(K)$ will stand for the canonical map into the direct limit.
By \cite{Halp}, there is an isomorphism  $\mathcal M_K \cong \ol{\varinjlim _n M_{p_n}(K)}$, where the completion is taken with respect to the unique rank function on the direct limit.
We henceforth will identify $\mathcal M = \mathcal M_K$ with the algebra $\ol{\varinjlim _n M_{p_n}(K)}$.

\begin{notation} Let $(X,d)$ be a metric space, $Y$ a subset of $X$
and $\varepsilon >0$. For $A\subseteq X$, we write
$A\subseteq_{\varepsilon} Y$ in case each element of $A$ can be
approximated by an element of $Y$ up to $\varepsilon$, that is, for
each $a\in A$ there exists $y\in Y$ such that $d(a,y) <
\varepsilon$.
\end{notation}

\begin{lemma}
\label{lem:sequences-and-morphisms-local}
Let $\calQ$ be a continuous factor with unique rank function $N_{\calQ}$. Assume
there exists a dense subalgebra $\calQ_0$ of $\calQ$ of countable dimension, and let $\{x_n\}_n$ be a $K$-basis of $\calQ_0$.

Let $\theta $ be a real number such that $0< \theta <1$. Assume further that we have constructed two strictly increasing
sequences $(q_i)$ and $(p_i)$ of natural numbers such that $p_i$ divides $p_{i+1}$, satisfying
$$1 > \frac{p_1}{q_1}>\cdots>\frac{p_i}{q_i}>\frac{p_{i+1}}{q_{i+1}}>\cdots > \theta , \qquad  \,\,\,  \lim_{i\to \infty} \frac{p_i}{q_i} = \theta ,$$
and
$$\frac{p_{i+1}}{q_{i+1}} - \theta < \frac{1}{2}\Big(\frac{p_i}{q_i} - \theta \Big) ,$$
for $i\ge 0$, where we set $p_0=q_0=1$.
Moreover, suppose that there exists a sequence of positive numbers $\varepsilon_i < \delta_i := \frac{p_i}{q_i} - \theta$ for each
$i$ (in which case $\delta_i < \frac{1}{2}\delta_{i-1}< 2^{-i}$ for all $i\ge 1$), and matricial subalgebras $\calA_i\subseteq \calQ$ together
with algebra homomorphisms $\rho_i : M_{p_i}(K) \rightarrow \calQ$ satisfying the following properties:
\begin{enumerate}[(i)]
\item $N_{\calQ}(\rho_i(1)) = \frac{p_i}{q_i}$ for all $i$,
\item For each $i$ and each $x \in \rho_i(1) \calA_i \rho_i(1)$, there exists $y \in M_{p_{i+1}}(K)$ such that
$$N_{\calQ}(x - \rho_{i+1}(y)) < \delta_i \, ,$$
\item For each $z \in M_{p_i}(K)$, we have
$$N_{\calQ}(\rho_i(z) - \rho_{i+1}(\gamma_{i+1,i}(z))) < \delta_i \, ,$$
\item $\mathrm{span}\{x_1,...,x_i\} \subseteq_{\varepsilon_i} \calA_i$ (that is, we can approximate every
element of $\mathrm{span} \{x_1,...,x_i\}$ by an element of
$\calA_i$ up to $\varepsilon_i$ in rank).
\end{enumerate}
Then there exists an isomorphism $\psi : \calM \rightarrow e \calQ e$, with $e \in \calQ$ an idempotent
such that $N_{\calQ}(e) = \theta$.
\end{lemma}

\begin{proof}
For a given positive integer $i$, and for $z \in M_{p_i}(K)$, we consider the sequence in $\calQ$
$$\{ \rho_j(\gamma_{j,i}(z)) \}_{j \geq i} .$$
It is a simple computation to show, using (iii), that for $h > j \geq i$ we have
\begin{equation}
\label{eq:Cauchyseq-local}
N_{\calQ}( \rho_h(\gamma_{h,i}(z)) - \rho_j(\gamma_{j,i}(z))) < \delta_j + \cdots + \delta_{h-1} \, .
\end{equation}
As a consequence, we obtain
$$N_{\calQ}( \rho_h(\gamma_{h,i}(z)) - \rho_j(\gamma_{j,i}(z))) <  2^{-j}+ \cdots + 2^{-h+1} < 2^{-j+1} ,$$
and the sequence is Cauchy; so we can define
$\psi_i : M_{p_i}(K) \rightarrow \calQ$ by
$$\psi_i(z) = \lim_j \rho_j(\gamma_{j,i}(z)) \in \calQ .$$
Note that $\psi_{i+1}(\gamma_{i+1,i}(z)) = \psi_i(z)$, so the maps $\{\psi_i\}_i$ give a well-defined
algebra homomorphism $\psi : \varinjlim_{i} M_{p_i}(K) \rightarrow \calQ$, defined by $\psi(\gamma_{\infty,i}(z)) = \psi_i(z)$ for $z \in M_{p_i}(K)$.

Observe that $N_{\calQ}(\rho_i(z)) = \frac{p_i}{q_i} N_{p_i}(z) = \frac{p_i}{q_i}N_{\calM}(\gamma_{\infty,i}(z))$ for $z \in M_{p_i}(K)$, where $N_{p_i}$ denotes the unique
rank function on $M_{p_i}(K)$. Therefore
$$N_{\calQ}(\psi(\gamma_{\infty,i}(z))) = \lim_j N_{\calQ}(\rho_j(\gamma_{j,i}(z))) = \lim_j \frac{p_j}{q_j} N_{\calM}(\gamma_{\infty,i}(z)) = \theta \cdot N_{\calM}(\gamma_{\infty,i}(z))$$
It follows that $N_{\calQ}(\psi(x)) = \theta \cdot N_{\calM}(x)$ for every $x \in \varinjlim_{i} M_{p_i}(K)$, and thus $\psi$ can be extended to a unital algebra
homomorphism $\psi : \calM \rightarrow e \calQ e$, where $e := \psi(1) = \varinjlim_{i} \rho_i(1)$, which satisfies the identity
$N_{\calQ}(\psi(z)) = \theta \cdot N_{\calM}(z)$ for all $z\in \calM$. In particular, $N_{\calQ} ( e ) = \theta $.
Clearly, $\psi$ is injective.

It remains to show that $\psi$ is surjective onto $e\calQ e$. Let $x \in \calQ$, and fix $\eta > 0$.
Take $i$ large enough so that
$$\varepsilon_i < \frac{\eta}{10}, \qquad \delta_i < \frac{\eta}{5}, \qquad N_{\calQ}(e - \rho_i(1)) < \frac{\eta}{5}$$
and such that there exists an element $\widetilde{x} \in $ span$\{x_1,...,x_i\}$ satisfying $N_{\calQ}(x-\widetilde{x}) < \frac{\eta}{10}$.

By (iv), there exists $y \in \calA_i$ so that $N_{\calQ}(\widetilde{x} - y) < \frac{\eta}{10}$; hence
$$N_{\calQ}(x - y) < \frac{\eta}{5} .$$
We thus have
\begin{align*}
& N_{\calQ}  (exe - \rho_i(1)y\rho_i(1)) \leq N_{\calQ}(exe - ex\rho_i(1)) + N_{\calQ}(ex\rho_i(1) - ey\rho_i(1)) \\
& + N_{\calQ}(ey\rho_i(1) - \rho_i(1) y \rho_i(1)) < \frac{3}{5} \eta .
\end{align*}
On the other hand, it follows from (ii) that there exists $z \in M_{p_{i+1}}(K)$ such that
$$N_{\calQ}(\rho_i(1)y\rho_i(1) - \rho_{i+1}(z)) < \delta_i < \frac{\eta}{5} \, .$$
Also, for $i+1<h$, we get from \eqref{eq:Cauchyseq-local}:
$$N_{\calQ}(\rho_h(\gamma_{h,i+1}(z)) - \rho_{i+1}(z)) < \delta_{i+1} + \cdots + \delta_{h-1} <\delta_{i+1} \Big(1+ \cdots + 2^{-h+i+2}\Big) <2\delta_{i+1} < \delta_i \, ,$$
and so letting $h \rightarrow \infty$ leads to
$$N_{\calQ}( \psi(\gamma_{\infty,i+1}(z)) - \rho_{i+1}(z)) \leq \delta_i < \frac{\eta}{5} \, .$$
Using the above inequalities, we obtain
\begin{align*}
&  N_{\calQ}(exe- \psi ( \gamma_{\infty,i+1}(z))) \le N_{\calQ}(exe- \rho_i(1) y \rho_i(1)) + N_{\calQ}( \rho_i(1) y \rho_i(1) - \rho_{i+1}(z)) \\
& + N_{\calQ}(\rho_{i+1}(z) - \psi ( \gamma_{\infty , i+1}(z))) \le  \frac{3}{5}\eta + \frac{1}{5}\eta + \frac{1}{5}\eta = \eta.
\end{align*}
By choosing a decreasing sequence $\eta_n \rightarrow 0$, it follows that there is a sequence $\{w_n\}_n$ of elements in $\calM$ such that $\lim_{n}\psi(w_n) = exe$.
Since $N_{\mathcal M}(w_i-w_j) = \theta^{-1} \cdot N_{\calQ}( \psi(w_i) - \psi(w_j))$ for all $i,j$, it follows that $\{w_n\}_n$ is a Cauchy sequence in $\calM$, and hence convergent to $w \in \calM$ satisfying $\psi(w) = exe$. This shows that $\psi$ is surjective.
\end{proof}

We now show how Theorem \ref{thm:local-ultra} follows from Lemma \ref{lem:sequences-and-morphisms-local}, assuming we are able to show that the hypotheses of that lemma are satisfied.
This indeed follows as in \cite{Halp}.
Take $\theta = 1/2$ and apply Lemma \ref{lem:sequences-and-morphisms-local} to obtain an isomorphism  $\psi\colon \mathcal M \to e\mathcal Q e$, where $N_{\calQ}(e)=1/2$.
Since $e\calQ \cong (1-e) \calQ $ as right $\calQ $-modules (by \cite[Corollary 9.16]{vnrr}), we get an isomorphism of $K$-algebras $\calQ \cong M_2 (e\calQ e)$.
Hence, we obtain an isomorphism $\mathcal M \cong M_2 (\mathcal M) \cong M_2 (e\calQ e)\cong \calQ$.

It remains to show that the hypotheses of Lemma \ref{lem:sequences-and-morphisms-local} are satisfied.
We need a preliminary lemma, which might be of independent interest.

\begin{lemma}
 \label{lem:prel-lemmaapprox}
 Let $p$ be a positive integer. Then there exists a constant $K(p)$, depending only on $p$, such that for any field $K$, for any $\varepsilon >0$, for any pair $\mathcal A \subseteq \mathcal B$,
 where $\mathcal B$ is a unital $K$-algebra and
 $\mathcal A$ is a unital
 regular subalgebra of $\mathcal B$, for any pseudo-rank function $N$ on $\mathcal B$, and for every $K$-algebra homomorphism $\rho \colon M_p(K) \to \mathcal B$
 such that $\{ \rho (e_{ij}) \mid i,j =1, \dots , p \} \subseteq_{\varepsilon}  \mathcal A  $ with respect to the $N$-metric, where
 $e_{ij}$ denote the canonical matrix units in $M_p(K)$, there exists a $K$-algebra homomorphism $\psi \colon M_p(K) \to \mathcal A$ such that
 $$N(\rho (e_{ij}) - \psi (e_{ij})) < K(p) \varepsilon \quad \text{ for } 1\le ,i,j\le p .$$
 \end{lemma}

\begin{proof}
 We proceed by induction on $p$. Let $p=1$, and let $K, \varepsilon, \mathcal A, \mathcal B, N$ and $\rho \colon K\to \mathcal B$ be as in the statement.
 Then $\rho (1) $ is an idempotent in $\mathcal B$ and, by assumption, there is $x\in \mathcal A$ such that $N(\rho (1) - x) <\varepsilon$. By
\cite[Lemma 19.3]{vnrr}, there exists an idempotent $g \in \calA$ such $x-g\in \mathcal A (x-x^2)$ and it follows that $N( \rho (1) -g) < 4\varepsilon $.
Therefore we can take $K(1) = 4$.

 Now assume that $p\ge 2$ and that there is a constant $K(p-1)$ satisfying the property corresponding to $p-1$.
  Let $K, \varepsilon, \mathcal A, \mathcal B, N$ and $\rho \colon M_p(K)\to \mathcal B$ be as in the statement.
  We identify $M_{p-1}(K)$ with the subalgebra of $M_p(K)$ generated by $e_{ij}$ with $1\le i,j \le p-1$. By the induction hypothesis, there is a set of $(p-1)\times (p-1)$
  matrix units $x_{ij}\in \mathcal A$ (so that $x_{ij}x_{kl}= \delta _{jk}x_{il}$ for $1\le i,j,k,l \le p-1$) satisfying $N(\rho (e_{ij}) - x_{ij})< K(p-1) \varepsilon $ for all
  $1\le i,j \le p-1$. By hypothesis, there are $z_{1p}, z_{p1} \in \mathcal A$ such that $N(\rho (e_{1p})- z_{1p}) < \varepsilon $ and
$N(\rho (e_{p1})- z_{p1}) < \varepsilon $. Our first task is to modify $z_{1p}$ and $z_{p1}$ in order to obtain new elements $z_{1p}'$ and $z_{p1}'$ such that
\begin{equation}
 \label{eq:orthogonal-first}
z_{1p}'x_{i1}= 0 = x_{1i}z_{p1}'\qquad  \text{for} \quad  1\le i \le p-1 ,
\end{equation}
with suitable bounds on the ranks. To get the desired elements, we proceed by induction on $i$. We will only prove the result for the position $(1,p)$.
The element in the position $(p,1)$ is built in a similar way. For $i=1$, we use that $\mathcal A$ is regular to obtain an idempotent $g_1\in \mathcal A$ such that
$$z_{1p}x_{11} \mathcal A = g_1 \mathcal A .$$
Note that
$$N(g_1) = N (z_{1p}x_{11}) = N(z_{1p}x_{11} - \rho (e_{1p})\rho (e_{11})) < \varepsilon + K(p-1) \varepsilon = (K(p-1)+1) \varepsilon .$$
Now take $z^{(1)}_{1p} : = (1-g_1) z_{1p}$. We get that $z^{(1)}_{1p} x_{11} = 0$ and that
$$N(z^{(1)}_{1p} -\rho (e_{1p}))  \le  N(z_{1p} - \rho (e_{1p})) + N(g_1) < (K(p-1)+2) \varepsilon .$$
Iterating this process, we get, for each $1\le i \le p-1$, an element $z^{(i)}_{1p}$ in $\mathcal A$ such that
$z^{(i)}_{1p} x_{j1} = 0$ for $1\le j \le i$ and
$$N( z^{(i)}_{1p} - \rho (e_{1p} ) )< ((2^{i}-1)K(p-1) + 2^i)\varepsilon .$$
Therefore, taking $z_{1p}'= z^{(p-1)}_{1p}$ and the element $z_{p1}':= z^{(p-1)}_{p1}$ built in a similar fashion, we get elements $z'_{1p},z_{p1}'\in \mathcal A$ satisfying \eqref{eq:orthogonal-first} and such that
\begin{equation}
 \label{eq:Nssecond}
\text{max} \{  N( z_{1p}' - \rho (e_{1p})) , N( z_{p1}' - \rho (e_{p1})) \}  <((2^{p-1}-1)K(p-1) + 2^{p-1}) \varepsilon .
 \end{equation}
The next step is to convert $z'_{1p}$ and $z_{p1}'$ into mutually quasi-inverse elements in $\mathcal A$. Indeed, we will replace in addition our original elements $x_{1i}$ and $x_{i1}$, for $1\le i \le p-1$, in order to get a coherent
family of ``partial matrix units'' $y_{1j}$, $y_{j1}$ for $1\le j\le p$. For this we will use \cite[Lemma 19.3]{vnrr} and its proof. Consider the element $x_{11}' := x_{11} z_{1p}'z_{p1}' x_{11}\in \calA$, and note that
\begin{align*}
N(x_{11}' - \rho (e_{11})) &  \le 2N(x_{11}-\rho (e_{11}))+ N(z_{1p}'- \rho (e_{1p}))+ N( z_{p1}'- \rho (e_{p1}))\\
& < (K(p-1)+1) 2^p \varepsilon ,
\end{align*}
where we have used the bound given by the induction hypothesis and \eqref{eq:Nssecond}. Therefore, we get
$$N(x_{11}'-(x_{11}')^2) <  3(K(p-1)+1) 2^p \varepsilon.$$
Now using \cite[Lemma 19.3]{vnrr} and its proof, we can find an idempotent $g\in \mathcal A$ such that
$g\in x_{11}' \mathcal A  \cap \calA x_{11}'$, $x_{11}'-g \in \calA (x_{11}'- (x_{11}')^2)$ and $x_{11}' g = g$. It follows that $g\le x_{11}$ and that $gx_{11}' g = g$, so that we get
$gz_{1p}'z_{p1}' g = g$. Set $y_{11} = g \le x_{11}$,  $y_{1i} = gx_{1i}$ and $y_{i1}= x_{i1}g$ for $i=2,\dots ,p-1$. Also, set $y_{1p}= gz_{1p}'$ and $y_{p1}= z_{p1}' g$.  Then we have
$y_{1i}y_{j1}= \delta_{ij} y_{11}$ for $1\le i,j \le p$. We have
\begin{align*}
N(g-\rho (e_{11})) & \le N(g-x_{11}') + N(x_{11}'-\rho (e_{11})) \\
& < 3 (K(p-1)+1) 2^p\varepsilon + (K(p-1)+1) 2^p\varepsilon = (K(p-1) +1) 2^{p+2} \varepsilon .
\end{align*}
Using this inequality, we obtain
\begin{align*}
  N(y_{1p}- \rho (e_{1p})) &   = N( gz_{1p}' - \rho (e_{11}) \rho (e_{1p})) \le N(g-\rho (e_{11})) + N( z_{1p}' - \rho (e_{1p})) \\
& < (K(p-1)+1)2^{p+2}\varepsilon + ((2^{p-1}-1)K(p-1) + 2^{p-1})\varepsilon \\
& = ((2^{p+2}+2^{p-1}-1)K(p-1) + 2^{p+2}+2^{p-1} )\varepsilon .
  \end{align*}
Similar computations give that $\text{max} \{ N(y_{1i}-\rho (e_{1i})), N(y_{i1} - \rho (e_{i1})): i=1,\dots, p \} <  ((2^{p+2}+2^{p-1}-1)K(p-1) + 2^{p+2}+2^{p-1} )\varepsilon $.
Finally, put $y_{ij} = y_{i1}y_{1j}$. We obtain that $\{ y_{ij} \}$ is a complete system of $p\times p$ matrix units in $\calA$, so that we can define a $K$-algebra homomorphism
$\psi \colon M_p(K) \to \mathcal A$ such that $\psi (e_{ij})= y_{ij}$. Moreover we have $N(\rho (e_{ij})-\psi (e_{ij}))< K(p) \varepsilon$, where
$K(p):= (2^{p+3}+2^p-2)K(p-1) + (2^{p+3}+ 2^p)$. This concludes the proof.
\end{proof}

We now show that that the hypotheses of Lemma \ref{lem:sequences-and-morphisms-local} are satisfied (assuming condition (iii) in Theorem \ref{thm:local-ultra}).
This is obtained from the next lemma by applying induction (starting with $p_0=q_0=1$ and $\calA _0=K$).

\begin{lemma}
 \label{lem:technical3}
Let $\calQ$ be a continuous factor with unique rank function $N_{\calQ}$. Assume
there exists a dense subalgebra $\calQ_0$ of $\calQ$ of countable dimension, and let $\{x_n\}_n$ be a $K$-basis of $\calQ_0$.
Assume that $\calQ$ satisfies condition (iii) in Theorem \ref{thm:local-ultra}, and let $\theta $ be a real number such that $0< \theta <1$.
Let $p$ be a positive integer such that there exist an algebra homomorphism $\rho \colon M_p(K) \to \mathcal \calQ$, a matricial subalgebra $\calA\subseteq \calQ$, a positive integer $m$, and $\varepsilon > 0$
 such that
\begin{enumerate}[(a)]
\item $N_{\calQ}(\rho(1)) = \frac{p}{q} > \theta$ for some positive integer $q$.
\item $\{ \rho(e_{ij}) \mid i,j=1,\dots , p \} \subseteq_{\varepsilon} \calA$, and $\mathrm{span} \{x_1,...,x_m\} \subseteq_{\varepsilon} \calA$.
\item $\varepsilon < \frac{1}{48 K(p) p^2} \Big( \frac{p}{q} - \theta \Big)$, where $K(p)$ is the constant introduced in Lemma \ref{lem:prel-lemmaapprox}.
\end{enumerate}
 Then there exist positive integers $p', g, q'$, with
 $p'=gp$, a real number $\varepsilon' > 0$, an algebra homomorphism $\rho' \colon M_{p'}(K) \to \calQ$ and a  matricial subalgebra $\calA' \subseteq \calQ$ such that the following conditions hold:
 \begin{enumerate}
  \item $N_{\calQ} (\rho' (1)) = p'/q' $.
  \item $$ 0 < \frac{p'}{q'} -\theta < \frac{1}{2} \Big( \frac{p}{q} -\theta \Big) .$$
  \item For each $x\in \rho (1) \calA \rho(1)$ there exists $y\in M_{p'}(K)$ such that
  $$N_{\calQ}(x - \rho ' (y)) < \frac{p}{q } - \theta .$$
  \item For each $z\in M_p (K)$, we have
  $$N_{\calQ } (\rho (z) - \rho'(\gamma (z)))< \frac{p}{q} -\theta ,$$
 where
$ \gamma \colon M_p(K) \to M_{p'}(K) = M_p(K)\otimes M_g(K)$
is the canonical unital homomorphism sending $z$ to $z\otimes 1_g$.
\item $\{ \rho'(e'_{ij}) \mid i,j=1,\dots , p' \} \subseteq_{\varepsilon'} \calA '$, and $\mathrm{span} \{x_1,...,x_m,x_{m+1}\} \subseteq_{\varepsilon'} \calA'$, where
$\{e'_{ij}\mid i,j=1,\dots , p' \}$ denote the canonical matrix
units in $M_{p'}(K)$.
\item $\varepsilon' < \frac{1}{48 K(p') {p'}^2} \Big( \frac{p'}{q'} - \theta \Big)$.
   \end{enumerate}
 \end{lemma}
\begin{proof}
 We denote by $e_{ij}$, for $1\le i,j \le p$, the canonical matrix units in $M_p(K)$. Set $f':= \rho (e_{11})$, which is an idempotent in $\calQ$
 with $N_{\calQ}(f')= 1/q $ (because $N_{\calQ}(\rho(1))= p/q $). By (b) and Lemma \ref{lem:prel-lemmaapprox}, there exists a $K$-algebra homomorphism $\psi \colon M_p(K)\to \mathcal A$ such
 that $N(\rho (e_{ij}) -\psi (e_{ij}))< K(p) \varepsilon $ for $1\le i,j\le p$.
 Set $f = \psi (e_{11})\in \calA $ and observe that
 \begin{equation}
\label{eq:first-approx}
 N_{\calQ}(f - f') < K(p) \varepsilon .
 \end{equation}
 Since $\calA$ is matricial, we can write $f=f_1+\cdots + f_k$, where $f_1,f_2,\dots ,f_k$ are nonzero mutually orthogonal
idempotents belonging to different simple factors of $\calA$. We can now consider, for each $1\le i \le k$, a set of matrix
units $\{ f^{(i)}_{jl} : 1\le j,l\le r_i \}$ inside $f_i\calA f_i$ such that each $f^{(i)}_{jj}$ is a minimal idempotent in the simple factor to which $f_i$ belongs, and such that
$$\sum_{j=1}^{r_i} f^{(i)} _{jj}= f_i$$
for $i=1,\dots ,k$.  Note that
 $$ N_{\calQ}(f) = \sum _{i=1}^{k} r_i N_{\calQ} (f^{(i)}_{11})$$
so that by \cite[Lemma 19.1(e)]{vnrr} and \eqref{eq:first-approx}, we get
\begin{equation}
\label{eq:local-ranks}
\Big| \frac{1}{q} - \sum _{i=1}^{k} r_i N_{\calQ} (f^{(i)}_{11}) \Big| = |N_{\calQ}(f') - N_{\calQ}(f)| \leq N_{\calQ}(f-f') < K(p) \varepsilon
\end{equation}
 We now approximate each real number $N_{\calQ}(f^{(i)}_{11})$ by a rational number $p_i/q_i$. Concretely,  we set
 $$\delta: = \frac{1}{ 48p (\sum_{i=1}^k r_i)} \Big( \frac{p}{q} -\theta\Big) \, ,$$
 and take positive integers $p_i, q_i$ so that $0< N_{\calQ}(f^{(i)}_{11}) - p_i/q_i < \delta$.
  Taking common denominator, we may assume that $q_i= q'$ for $i=1,\dots ,k $.
 Let $\alpha '$ be such that $1/\alpha ' = \sum _{i=1}^{k} r_i p_i/q' $, and observe that, by using \eqref{eq:local-ranks}, we have
\begin{align*}
 | p/q - p/\alpha' | & \le  p \Big[ \Big | \frac{1}{q} - \sum_{i=1}^k r_i N_{\calQ}(f^{(i)}_{11}) \Big |  + \sum_{i=1}^k r_i \Big( N_{\calQ}(f^{(i)}_{11}) - \frac{p_i}{q'}\Big) \Big] \\
& < K(p) p \epsilon + p \Big( \sum_{i=1}^k r_i \Big) \delta < \frac{1}{48p}(p/q - \theta ) + \frac{1}{48}(p/q - \theta) \leq \frac{1}{24}(p/q - \theta).
\end{align*}
So in particular
\begin{equation}
\label{eq:1over8-inequ:2}
| p/q - p/\alpha' | < \frac{1}{8}(p/q - \theta )
\end{equation}
 Now take
 $$ \lambda _i = \frac{\alpha ' p_i r_i }{q'} \qquad \text{and} \qquad \varepsilon _i = \frac{\lambda_i}{pr_i}\Big( \frac{p}{q} - \theta \Big).$$
 Then $\sum_{i=1}^k \lambda _i =1$. Moreover $\lambda_i$, $\varepsilon_i$ (and of course $p_i/q'$), $i=1,\dots ,k$, do not depend on replacing $p_i$ and $q'$ by $p_iN$
and $q'N$ respectively, for any integer $N\ge 1$, so we can assume that $p_i$ and $q'$ are arbitrarily large.
 Taking $p_i$ large enough, we see that we can find non-negative integers $p_i'$, for $1\le i\le k$, such that
 \begin{equation}
  \label{eq:3over4-inequ:2}
\frac{5\varepsilon_i}{8} < \frac{p_i}{q'} - \frac{p_i'}{q'} < \frac{3\varepsilon_i}{4}
  \end{equation}
 for $i=1\dots ,k $. Indeed, using \eqref{eq:1over8-inequ:2}, we can see that
 $$\frac{5\varepsilon_i q'}{8} = \frac{5p_i}{8}\Big( \frac{\alpha'}{p} \Big( \frac{p}{q} -\theta \Big) \Big) < \frac{5p_i}{8} \frac{8}{7} = \frac{5}{7} p_i < p_i \, .$$
 We can choose $q'$ big enough so that $\frac{\varepsilon_i q'}{8}>1$ and thus there is an integer in the open interval $( \frac{5}{8}\varepsilon_i q' , \frac{3}{4}\varepsilon_i q' )$. Since
 $\frac{5\varepsilon_i q'}{8} <p_i$ we can find a {\it non-negative} integer $p_i'$ such that  \eqref{eq:3over4-inequ:2} holds.

 Now, using  that $\sum _{i=1}^k \lambda_i =1$, we get
 $$\frac{5}{8p} \Big( \frac{p}{q} - \theta \Big) < \frac{1}{\alpha'} - \Big( \sum _{i=1}^k r_i \frac{p_i'}{q'}\Big) <  \frac{3}{4p} \Big( \frac{p}{q} - \theta \Big) .$$
 Hence, setting $g= \sum _{i=1}^k r_i p_i'$ and $p '= pg$, we get
 $$ \frac{5}{8} \Big( \frac{p}{q} - \theta \Big) < \frac{p}{\alpha'} - \frac{p'}{q'} <  \frac{3}{4}\Big( \frac{p}{q} - \theta \Big) .$$
 Hence, using \eqref{eq:1over8-inequ:2}, we get
\begin{equation}
\label{eq:key-inequal1:2}
 \frac{1}{2} \Big( \frac{p}{q} - \theta \Big) < \frac{p}{q} - \frac{p'}{q'} <  \frac{7}{8}\Big( \frac{p}{q} - \theta \Big) .
\end{equation}
   Now, since $\calQ$ is continuous, there exists an idempotent $e$ in $\calQ$ such that $N_{\calQ} (e)= 1/q'$, and since $p_i'/q' <N_{\calQ} (f_{11}^{(i)})$ and
   $\calQ$ is simple and injective, we get
 $$ p_i' \cdot e \lesssim f_{11}^{(i)} $$
 for $i=1,\dots , k$. We may (and will) assume that $e\le f^{(1)}_{11}$.  Therefore we can build a system of matrix units
 $$ \{ h^{(i_1,i_2)}_{( j_1,j_2), (u_1,u_2) } : 1\le i_1,i_2\le k,\,\,  1\le j_1 \le r_{i_1}, \, \, 1\le j_2\le r_{i_2} , \, \, 1\le u_1\le p'_{i_1} , \, \,  1\le u_2 \le p'_{i_2} \}, $$
  with
 $$h^{(i_1,i_2)}_{(j_1,j_2), (u_1,u_2)} h^{(i_1',i_2')}_{(j_1',j_2'), (u_1',u_2')} =\delta _{i_2,i_1'}\cdot \delta_{j_2,j_1'}\cdot \delta_{u_2,u_1'} \cdot h^{(i_1,i_2')}_{(j_1,j_2'),(u_1,u_2')}\, ,$$
 such that
$e= h^{(1,1)}_{(1,1), (1,1)}$, $\{ h^{(i,i)}_{(1,1),(u_1,u_2)} : 1\le u_1,u_2\le p_i' \}$ is a system of matrix units inside
$f^{(i)}_{11} \calQ f^{(i)}_{11}$ for all $i$,
and
$$h^{(i,i)}_{(j_1,j_2),(u_1,u_2)} = f^{(i)}_{j_1, 1} h^{(i,i)}_{(1,1),(u_1,u_2)} f^{(i)}_{1,j_2}$$
for $i=1,\dots , k$, $1\le j_1,j_2\le r_i$ and $1\le u_1,u_2\le p_i'$.
To build such a system of matrix units, we proceed as follows. First we construct a family
$$\{ h^{(1,i)}_{(1,1), (1,u)}, h^{(i,1)}_{(1,1), (u, 1)} : 1\le u \le p_i' \}$$
so that $$h^{(1,i)}_{(1,1),(1,u)}h^{(i,1)}_{(1,1),(u,1)}= h^{(1,1)}_{(1,1),(1,1)}= e , \qquad e h^{(1,i)}_{(1,1),(1,u)} = h^{(1,i)}_{(1,1),(1,u)} , \qquad
h^{(i,1)}_{(1,1),(u,1)} e= h^{(i,1)}_{(1, 1),(u, 1)} ,$$
and such that $h^{(i,i)}_{(1,1), (u,u)} : = h^{(i,1)}_{(1,1), (u,1)} h^{(1, i)}_{(1,1), (1,u)}$
are pairwise orthogonal idempotents inside $f^{(i)}_{11}\calQ f^{(i)}_{11}$ for all $i$.
Then, define, for $1\le i \le k$, $1\le j\le r_i$ and $1\le u\le p_i'$,
$$h^{(i,1)}_{(j,1),(u,1)} = f^{(i)}_{j,1} \cdot  h^{(i,1)}_{(1,1), (u,1)}, \qquad h^{(1,i)}_{(1,j),(1,u)}= h^{(1,i)}_{(1,1),(1,u)}\cdot f^{(i)}_{1,j}.$$
Finally, set, for $1\le i_1,i_2 \le k$, $1\le j_1\le r_{i_1}$, $1\le j_2\le r_{i_2}$, $1\le u_1\le p'_{i_1}$, $1\le u_2\le p'_{i_2}$,
$$h^{(i_1,i_2)}_{(j_1,j_2),(u_1,u_2)} = h^{(i_1,1)}_{(j_1,1),(u_1,1)} \cdot h^{(1,i_2)}_{(1,j_2),(1,u_2)}.$$
It is straightforward to verify that the family $\{ h^{(i_1,i_2)}_{(j_1,j_2),(u_1,u_2)} \}$ satisfies the required properties.

Recalling that $g = \sum _{i=1}^k r_ip_i'$, we get that $\{ h^{(i_1,i_2)}_{(j_1,j_2),(u_1,u_2)} \}$ is a system of $g\times g$-matrix units
inside $f \calQ f$, and we can now define an algebra homomorphism $\rho'\colon M_{p'}(K)= M_p(K)\otimes M_g(K)\to \mathcal \calQ$ by the rule
$$\rho ' (e_{ij}\otimes e^{(i_1,i_2)}_{(j_1,j_2),(u_1,u_2)}) = \psi (e_{i1})  h^{(i_1,i_2)}_{( j_1,j_2), (u_1,u_2) } \psi (e_{1j}) \, ,$$
where $\{ e^{(i_1,i_2)}_{(j_1,j_2),(u_1,u_2)} \}$ is a complete system of matrix units in $M_g(K)$.

 It remains to verify properties (1)-(7).

 (1) We have
 $$N_{\calQ} (\rho ' (1)) = pg N_{\calQ} (e) = \frac{p'}{q'},$$
 as desired.

 (2) This follows from \eqref{eq:key-inequal1:2}.

 (3) Let $x\in \rho(1)  \calA \rho (1)$. Then we can write
 $$ x = \rho(1) x' \rho(1) = \sum_{a,b=1}^p \rho (e_{a1}) f' \rho(e_{1a}) x' \rho(e_{b1}) f' \rho (e_{1b}) \, ,$$
 where $x' \in \calA$. Now by $(b)$ we can approximate each $\rho(e_{1a}), \rho(e_{b1})$ by an  element of $\calA$:
$$N_{\calQ}(\rho(e_{1a}) - x_{1a}) < \varepsilon \qquad {\mathrm {for} } \quad x_{1a} \in \calA$$
$$N_{\calQ}(\rho(e_{b1}) - x_{b1}) < \varepsilon \qquad {\mathrm {for} } \quad x_{b1} \in \calA$$
Thus we can consider the element
$$\widetilde{x} = \sum_{a,b=1}^p \psi (e_{a1}) \underbrace{f x_{1a} x' x_{b1} f}_{x_{ab} \in f \calA f} \psi (e_{1b}) \in \calA \, ,$$
so that $N_{\calQ}(x-\widetilde{x}) < p^2(4K(p) +2)\varepsilon < 5K(p)p^2\varepsilon$ (using that $K(p)\ge 4$ for the last inequality). Now we can write each $x_{ab}$ in the form
 $$ x_{ab} = \sum _{i=1}^k \sum_{j,l=1}^{r_i} \lambda(a,b)^{(i)}_{jl} f^{(i)}_{jl}$$
 for some scalars $\lambda (a,b)^{(i)}_{jl}\in K$.

Take
 $$y = \sum _{a,b=1}^p e_{ab} \otimes \Big( \sum_{i=1}^k \sum_{j,l=1}^{r_i} \lambda(a,b)^{(i)}_{jl}\Big( \sum_{u=1}^{p_i'} e^{(i,i)}_{(j,l), (u,u)}\Big) \Big)\, \in M_{p'}(K) .$$
Writing $c_{bil}= \sum_{a=1}^p \sum_{j=1}^{r_i} \psi (e_{a1}) \lambda (a,b)_{jl}^{(i)} f_{j1}^{(i)}$ and $d_{bil} = f_{1l}^{(i) } \psi (e_{1b})$,  we have
 \begin{align*}
  N_{\calQ}(\widetilde{x} - \rho '(y)) &  = N_{\calQ} \Big( \sum_{b=1}^p \sum_{i=1}^k \sum_{l=1}^{r_i} c_{bil} \Big( f_{11}^{(i)} - \sum_{u=1}^{p_i'} h^{(i,i)}_{(1,1),(u,u)} \Big) d_{bil} \Big) \\
&  \le p\Big( \sum_{i=1}^k r_i \Big( N_{\calQ} (f_{11}^{(i)}) - N_{\calQ} \Big( \sum_{u=1}^{p_i'} h^{(i,i)}_{(1,1), (u,u)}\Big) \Big) \Big) \\
  & = p\Big( \sum_{i=1}^k r_i \Big( N_{\calQ}(f^{(i)}_{11})- \frac{p_i'}{q'} \Big)  \Big) \\
    & = p\Big( \sum_{i=1}^k r_i N_{\calQ}(f^{(i)}_{11}) \Big) - p \frac{ \sum_{i=1}^k r_i p_i'}{q'} \\
    & = p\Big[  \sum_{i=1}^k r_i N_{\calQ}(f^{(i)}_{11}) - \frac{1}{q} \Big] + \Big(\frac{p}{q} - \frac{p'}{q'} \Big)
    \end{align*}
Using \eqref{eq:local-ranks} and \eqref{eq:key-inequal1:2}, we get
\begin{align*}
  N_{\calQ}(\widetilde{x} - \rho '(y)) & < K(p) \varepsilon p + \frac{7}{8}\Big( \frac{p}{q} - \theta \Big) \\
    & \leq \frac{1}{48p}\Big( \frac{p}{q} - \theta \Big) + \frac{7}{8} \Big( \frac{p}{q} - \theta \Big) \le  \frac{43}{48} \Big(\frac{p}{q} - \theta \Big)
\end{align*}
Putting everything together,
$$N_{\calQ}(x - \rho'(y)) < 5K(p) p^2 \varepsilon + \frac{43}{48} \Big(\frac{p}{q} - \theta \Big) < \frac{5}{48} \Big(\frac{p}{q} - \theta \Big) + \frac{43}{48}\Big(\frac{p}{q} - \theta \Big) = \frac{p}{q} - \theta .$$

 (4) Suppose now that $x= \rho (z)$ for $z = \sum_{a,b=1} ^p \mu _{ab} e_{ab}\in M_p(K)$, with $\mu_{ab}\in K$. Then we can see that, using the same construction of $y$ given in (3), we obtain
 $$ y =  z\otimes \sum_{i=1}^k \sum _{j=1}^{r_i} \sum_{u=1}^{p_i'} e^{(i,i)}_{(j,j), (u,u)} = \gamma (z).$$

To conclude the proof, just take $\varepsilon' > 0$ satisfying
$\varepsilon ' < \frac{1}{48 K(p') p'^2}\Big( \frac{p'}{q'} -
\theta\Big)$ and, using condition (iii) in Theorem
\ref{thm:local-ultra}, consider a matricial subalgebra $\calA'$ such
that $\{ \rho' (e'_{ij}) \mid i,j=1,\dots , p' \}
\subseteq_{\varepsilon'} \calA '$ and $\mathrm{span}
\{x_1,...,x_m,x_{m+1}\} \subseteq_{\varepsilon'} \calA'$.
\end{proof}

\section{$D$-rings}
\label{sect:Drings}

We now consider a generalization of Theorem \ref{thm:local-ultra} to $D$-rings, where $D$ is a division ring. We have not found
a reasonable analogue of the local condition (iii) in this setting, but we are able to extend condition (ii).

The reason we consider this generalization is the question raised by Elek and Jaikin-Zapirain of whether the completion of the $*$-regular closure of the group algebra
of a countable amenable ICC-group is isomorphic to either $M_n(D)$, $n\ge 1$, or to $\mathcal M _D$, for some division ring $D$. Here $\mathcal M _D$ is the completion of $\varinjlim _n M_{2^n}(D)$
with respect to its unique rank function.

Throughout this section, $D$ will denote a division ring, and $K$ will stand for the center of $D$.

A $D$-ring is a unital ring $R$ together with a unital ring homomorphism $\iota \colon D\to R$. A morphism of $D$-rings $R_1\to R_2$ is a ring homomorphism $\varphi \colon R_1\to R_2$
such that $\iota_2 = \varphi \circ \iota_1$. A {\it matricial $D$-ring} is a $D$-ring $\mathcal A$ which is isomorphic as a $D$-ring to a finite direct product
$M_{n_1}(D)\times \cdots \times M_{n_r}(D)$, where the structure of $D$-ring of the latter is the canonical one. A $D$-ring $\calA$ is an {\it ultramatricial $D$-ring} if
it is isomorphic as $D$-ring to a direct limit of a sequence $(\mathcal A_n, \varphi_n)$ of matricial $D$-rings $\mathcal A_n$ and $D$-ring homomorphisms
$\varphi_n\colon \mathcal A_n \to \mathcal A_{n+1}$.

We start with a simple lemma.

\begin{lemma}
 \label{lem:unique-in-tensor} There is a unique rank function $S$ on the (possibly non regular) simple $D$-ring $D\otimes _K \calM_K$, and
 $D\otimes_K (\varinjlim_n M_{2^n}(K))\cong \varinjlim_n M_{2^n} (D)$ is dense in $D\otimes _K \calM_K$ with respect to the $S$-metric.
\end{lemma}

\begin{proof} We denote by $N_{\calM_K}$ the unique rank function on $\calM_K$.

 The ring $D\otimes _K \calM_K$ is simple by \cite[Corollary 7.1.3]{Cohn}. Let $x= \sum_{i=1}^k d_i\otimes x_i \in D\otimes _K \calM_K$ and $\varepsilon >0$. Let $y_i \in \varinjlim_n M_{2^n}(K)$ be such that
 $N_{\calM_K} (x_i-y_i) < \frac{\varepsilon}{k}$, and set $y:= \sum_{i=1}^k d_i\otimes y_i$. Then, for any rank function $S$ on $D\otimes _K \calM_K$, we have
 $$S(x- y) \le \sum _{i=1}^k S(1\otimes (x_i-y_i) ) = \sum_{i=1}^k N_{\calM_K} (x_i-y_i) < \varepsilon .$$
Since there is a unique rank function on $\varinjlim_n M_{2^n}(D)$, this shows at once that there is a  unique rank function $S$ on $D\otimes _K \calM_K$, and that
$D\otimes_K (\varinjlim_n M_{2^n}(K))$ is dense in $D\otimes _K \calM_K$ with respect to the $S$-metric.
\end{proof}

 \begin{theorem}
 \label{thm:ultra-division}
Let $\calA$ be an ultramatricial $D$-ring, and let $N$ be an extremal pseudo-rank function on $\calA$ such that the completion $\calQ$ of $\calA$ with respect to $N$ is
a continuous factor. Then there is an isomorphism of $D$-rings $\calQ\cong \mathcal M_D$.
\end{theorem}

\begin{proof} We can assume that  $\calA = \varinjlim_n (\calA_n, \varphi_n)$, where each $\calA_n$ is a matricial $D$-ring and each map $\varphi_n \colon
\calA_n\to \calA_{n+1}$ is an injective morphism of $D$-rings.

Write $\mathcal B_n = C_{\calA_n} (D)$ for the centralizer of $D$ in $\calA_n$. Then $\calB_n$ is a matricial $K$-algebra and, since $K$ is the center of $D$, we have
$\calA_n \cong D\otimes_K \calB_n$. Moreover, we have $\varphi_n (\calB_n) \subseteq \calB_{n+1}$ for all $n\ge 1$, and $\calA \cong D\otimes _K \calB$, where
$\calB = C_{\calA} (D)= \varinjlim_n (\calB_n , (\varphi_n)|_{\calB_n})$ is an ultramatricial $K$-algebra.

Now, it is not hard to show that the restriction map $S\to S_{\calB}$ defines an affine homeomorphism
$\mathbb P (\calA ) \cong \mathbb P (\calB)$. Consequently, the restriction $N_\calB$ of $N$ to $\calB$ is an extremal pseudo-rank function on $\calB$.
Moreover, since $N(\calA )= N(\calB)$, it follows that $N(\calB)$ is a dense subset of the unit interval, which implies that the completion $\mathcal R$
of $\calB$ in the $N_\calB$-metric is a continuous factor over $K$. Now it follows from Theorem \ref{thm:local-ultra} that there is a $K$-algebra
isomorphism $\psi' \colon \calM_K  \to \mathcal R$, which induces an isomorphism of $D$-rings
$$\psi := \text{id}_D \otimes \psi' \colon D\otimes _K \calM_K \to D\otimes _K \mathcal R.$$
Since $\calA \subseteq   D\otimes _K \mathcal R = \psi (D\otimes_K \calM_K) \subseteq \calQ$, it follows that $\psi (D\otimes_K \calM_K)$ is dense
in $\calQ$. By Lemma \ref{lem:unique-in-tensor}, $\psi (D\otimes _K (\varinjlim_n M_{2^n}(K)))$ is dense in  $\psi (D\otimes_K \calM_K)$ with respect to the restriction
of $N_{\calQ}$ to it, therefore $\psi (D\otimes _K (\varinjlim_n M_{2^n}(K)))$ is dense in $\calQ$. Hence, the restriction of $\psi$ to
$D\otimes_K (\varinjlim_n M_{2^n}(K)) \cong \varinjlim_n M_{2^n}(D)$ gives a rank-preserving isomorphism of $D$-rings from $\varinjlim_n M_{2^n}(D)$ onto a dense
$D$-subring of $\calQ$, and thus it can be uniquely extended to an isomorphism from $\calM_D$ onto $\calQ$.
 \end{proof}

\section{Fields with involution}

In this section, we will consider the corresponding problem for $*$-algebras. Again, the motivation comes from the theory of group algebras. If
$K$ is a subfield of $\C$ closed under complex conjugation, and $G$ is a countable discrete group, then there is a natural involution on the group algebra $K[G]$,
and the completion of the $*$-regular closure of $K[G]$ in $\mathcal U (G)$ is a $*$-regular ring containing $K[G]$ as a $*$-subalgebra.
It would be thus desirable to find conditions under which this completion is $*$-isomorphic to $\calM_K$, where $\calM_K$ is endowed with the involution induced from
the involution on $\varinjlim_n M_{2^n}(K)$, which is in turn obtained by endowing each algebra $M_{2^n}(K)$ with the conjugate-transpose involution.

We recall some facts about $*$-regular rings and their completions (see for instance \cite{A87, Berb}). A {\it $*$-regular ring}
is a regular ring endowed with a proper involution, that is, an involution $*$ such that $x^*x= 0$ implies $x=0$. The involution is called {\it positive definite}
in case the condition
$$\sum_{i=1}^n x_i^* x_i= 0 \implies x_i=0 \qquad   (i=1,\dots , n) $$
holds for each positive integer $n$. If $R$ is a $*$-regular ring with positive definite involution, then $M_n(R)$, endowed with the $*$-transpose
involution, is a $*$-regular ring.

We will work with $*$-algebras over a field with positive definite involution $(F, *)$.  The involution on $M_n(F)$ will always be the $*$-transpose involution.

The $*$-algebra $A$ is {\it standard matricial} if $A= M_{n(1)}(F)\times \cdots \times M_{n(r)}(F)$ for some positive integers $n(1),\dots , n(r)$.
A {\it standard map} between two standard matricial $*$-algebras $A$, $B$ is a block-diagonal $*$-homomorphism $A\to B$ (see \cite[p. 232]{A87}). A {\it standard ultramatricial}
$*$-algebra is a direct limit of a sequence $A_1\xrightarrow{\Phi_1} A_2 \xrightarrow{\Phi_2} A_3 \xrightarrow{\Phi_3}  \cdots $ of standard matricial $*$-algebras
$A_n$ and standard maps $\Phi_n\colon A_n\to A_{n+1}$. An {\it ultramatricial} $*$-algebra is a $*$-algebra which is $*$-isomorphic to the direct limit of a sequence
of standard matricial $*$-algebras $A_n$ and $*$-algebra maps $\Phi_n\colon A_n\to A_{n+1}$.
Let $\calA$ be a $*$-algebra which is $*$-isomorphic to a standard matricial algebra, through a $*$-isomorphism
$$\gamma \colon \calA \to M_{n(1)}(F)\times \cdots \times M_{n(r)}(F).$$
Then we say that a projection (i.e., a self-adjoint
idempotent) $p$ in $\calA$ is {\it standard} (with respect to $\gamma$) in case, for each $i=1,2,\dots,r$,  the $i$-th component
$\gamma (p)_i$ of $\gamma (p)$ is a diagonal projection in $M_{n(i)}(F)$.

Two idempotents $e,f\in R$ are {\it equivalent}, written $e\sim f$, if there are $x\in eRf$ and $y\in  fRe$ such that $e=xy$, $f=yx$. If $e,f$ are projections of a $*$-ring $R$, then we say that $e$ is $*$-equivalent to $f$, written $e\overset{*}{\sim} f$,  in case there is $x\in eRf$ such that
$e=xx^*$ and $f=x^*x$.

If $R$ is a $*$-regular ring and $x\in R$, then there exist unique projections ${\rm LP}(x)$ and ${\rm RP}(x)$, called the {\it left} and the {\it right} projections of $x$, such that $xR= {\rm LP}(x) R$ and $Rx= R\cdot {\rm RP}(x)$.
Moreover, with $e={\rm LP}(x)$ and $f = {\rm RP}(x)$, there exists a unique element $y\in R$, the {\it relative inverse} of $x$, such that $xy=e$ and $yx=f$. We will denote the relative inverse of $x$ by $\ol{x}$.

A $*$-regular ring $R$ satisfies the condition $\LPeqRP$ in case ${\rm LP}(x)  \overset{*}{\sim} {\rm RP}(x)$  holds for each $x\in R$. Observe that $R$ satisfies $\LPeqRP$ if and only if equivalent projections of $R$ are $*$-equivalent (\cite[Lemma 1.1]{A87}).
In general this condition is not satisfied for a $*$-regular ring, but many $*$-regular rings satisfy it. It is worth to mention that for a field $F$ with positive definite involution,
$M_n(F)$ satisfies $\LPeqRP$ for all $n\ge 1$ if and only if $F$ is $*$-Pythagorean \cite[Theorem 4.9]{HandRocky} (see also \cite[Theorem 4.5]{HandTAMS}, \cite[Theorem 1.12]{A87}).

The following result is relevant for our purposes:

\begin{theorem}\cite[Theorem 3.5]{A87}
 \label{thm:ultra-LP}
 Let $(F,*)$ be a field with positive definite involution, let $\mathcal A$ be a standard ultramatricial $*$-algebra, and let $N$ be a pseudo-rank function on $\calA$. Then the type II part of the $N$-completion
of $\calA$ is a $*$-regular ring satisfying $\LPeqRP $.
 \end{theorem}

As a consequence of this result, the $*$-algebra $\calM_F$ always satisfies $\LPeqRP$, independently of whether the field $F$ is $*$-Pythagorean or not.

We collect, for the convenience of the reader, some properties of a pseudo-rank function on a $*$-regular ring. For an element $r$ of a $*$-regular ring $R$ we denote by $\ol{r}$
the {\it relative inverse} of $r$.

\begin{lemma}
 \label{lem:prpr-N_Star}
 Let $N$ be a pseudo-rank function on a $*$-regular ring $R$. The following hold:
 \begin{enumerate}[(a)]
  \item The involution is isometric, that is, $N(r^*)= N(r)$ for each $r\in R$.
  \item $N(\ol{r} - \ol{s})\le 3N(r-s)$ for all $r,s\in R$.
  \item $N(\mathrm{LP} (r) - \mathrm{LP}(s)) \le 4N(r-s)$, and $N(\mathrm{RP}(r)- \mathrm{RP}(s))\le 4 N(r-s)$  for all $r,s\in R$.
   \item Suppose that $e_1,e_2,f_1,f_2$ are projections in $R$ such that $f_1\overset{*}{\sim} f_2$ and $N(e_i-f_i)\le \varepsilon $ for $i=1,2$.
  Then there exist projections $e_i'\le e_i$ such that $e_1'\overset{*}{\sim} e_2'$ and $N(e_i-e_i') \le 5 \varepsilon $ for $i=1,2$.
 \end{enumerate}
\end{lemma}

\begin{proof}
 (a) See the proof of Proposition 1 in \cite{HandCMB} or \cite[Proposition 6.11]{Jaikin-survey}.

 (b) In \cite[p. 310]{BH}, it is shown that $N(\ol{r}-\ol{s}) \le 19 N(r-s)$, and the authors comment that K. R. Goodearl has reduced $19$ to $5$. Here we show that indeed it can be reduced
 to $3$.

 Let $e= r\ol{r}$, $f= \ol{r} r$, $g= s\ol{s}$ and $h= \ol{s} s$. Clearly $r^*r+(1-f)$ and $ss^*+(1-g)$ are invertible in $R$ and so
 \begin{equation}
  \label{eq:rankequal}
  N(\ol{r}-\ol{s}) = N((r^*r+(1-f))(\ol{r}-\ol{s})(ss^*+(1-g))).
  \end{equation}
  On the other hand, we have
  \begin{align*}
   & N((r^*r+(1-f))(\ol{r}-\ol{s})(ss^*+(1-g))) = N( r^*ss^*-r^*rs^* + r^*(1-g) - (1-f) s^*) \\
   & \le N( r^*ss^* - r^*rs^*) + N(r^*(1-g) - (1-f) s^*) \\
   & \le N(s-r) + N(r^*-s^*) + N( fs^*-r^*g) \\
   & = 2N(s-r) + N( f(s^*-r^*)g) \le 3 N(r-s).
     \end{align*}
By \eqref{eq:rankequal}, we get $N(\ol{r}- \ol{s}) \le 3 N(r-s)$.

(c) Using (b), we get
\begin{align*}
N(\mathrm{RP}(r)- \mathrm{RP}(s)) & = N(\ol{r}r - \ol{s}s) \le N((\ol{r}-\ol{s})r) + N(\ol{s}(r-s)) \\
& \le 3N(r-s) + N(r-s) = 4 N(r-s).
\end{align*}

The proof for ${\rm LP}$ is similar.

(d) We follow the idea in \cite[proof of Lemma 2.6]{A87}. Let $w\in f_1 R f_2$ be a partial isometry such that $f_1= ww^*$
 and $f_2= w^*w$. Consider the self-adjoint element
$a= e_1- e_1ww^*e_1$ and set $p_1:= {\rm LP}(a) = {\rm RP}(a)\le e_1$. Then
$$N(p_1) = N(a) = N( e_1-e_1f_1e_1) \le \varepsilon. $$
Set $p_1':= e_1 - p_1$. Then $N(e_1-p_1')\le \varepsilon $ and, since $p_1'ap_1'= 0$, we have  $p_1'= w'(w')^*$, where $w':= p_1'w$.

Now observe that $(w')^* w' = w^*p_1'w \le w^*w = f_2$. Consider the elements
$$b= e_2 - e_2 (w')^*w'e_2, \qquad e_2'' = {\rm LP}(b) = {\rm RP}(b) .$$
We  have $e_2'' \le e_2$ and
\begin{align*}
N(e_2'') & = N(b)= N(e_2- e_2(w')^*w'e_2) \le N(e_2-(w')^*w') \\
& \le N(e_2 - f_2) +N(w^*w-w^*p_1'w)\\
 & \le \varepsilon + N(f_1-p_1')\\
 & \le \varepsilon + N(f_1-e_1) + N(e_1-p_1') \le 3 \varepsilon.
\end{align*}
Set $e_2'= e_2-e_2''$. As before, we have $e_2' = (w'')^*w''$, where $w''= w'e_2'= p_1'we_2'$, and $N(e_2-e_2') =N(e_2'') \le 3 \varepsilon$.
Write $e_1' = w''(w'')^*$. Then $e_1' \le p_1' \le e_1$, $e_1' \overset{*}{\sim} e_2'$, and
\begin{align*}
 N(e_1-e_1') & = N(e_1-p_1') + N(p_1'- e_1') \le \varepsilon + N(w' f_2(w')^*- w'e_2'(w')^*) \\
 & \le \varepsilon + N(f_2-e_2)+ N(e_2-e_2') \le 5\varepsilon.
\end{align*}
 \end{proof}

\begin{lemma}
 \label{lem:LPRPapproximated}
 Let $R$ be a $*$-regular ring, and assume that $R$ is complete with respect to a rank function $N$. Then $R$ satisfies $\LPeqRP$ if and only if, given equivalent projections $p,q\in R$ and $\varepsilon >0$ there exist
 subprojections $p'\le p$ and $q'\le q$ such that $p'\overset{*}{\sim} q'$ and $N(p-p')<\varepsilon$, $N(q-q')< \varepsilon$.
 \end{lemma}

 \begin{proof}
 The ``only if'' direction follows trivially from \cite[Lemma 1.1]{A87}.

To show the ``if'' direction, suppose that $p$ and $q$ are equivalent projections of $R$. Assume we have built, for some $n\ge 1$, orthogonal projections $p_1,\dots ,p_n\le p$ and $q_1,\dots , q_n \le q$ such that
  $p_i\overset{*}{\sim} q_i$ for $i=1,\dots , n$, and $N(p-(\sum_{i=1}^n p_i))< 2^{-n}$, $N(q-(\sum_{i=1}^n q_i))< 2^{-n}$. Set $p':= p-(\sum_{i=1}^n p_i)$ and $q'= q-(\sum_{i=1}^n q_i)$.
  Then $p'\sim q'$ by \cite[Theorems 19.7 and 4.14]{vnrr}, so that there are subprojections $p_{n+1} \le p'$ and $q_{n+1}\le q'$ such that $p_{n+1}\overset{*}{\sim} q_{n+1}$
  and $N(p'-p_{n+1}) < 2^{-n-1}$ and $N(q'-q_{n+1}) < 2^{-n-1}$. Therefore we can build sequences $\{p_n \}$ and $\{ q_n \}$ of orthogonal subprojections of $p$ and $q$ respectively
  such that $p_n\overset{*}{\sim} q_n$, and $N(p-(\sum_{i=1}^n p_i))< 2^{-n}$, $N(q-(\sum_{i=1}^n q_i))< 2^{-n}$ for all $n\ge 1$.
  Let $w_n\in p_nRq_n$ be partial isometries such that $p_n=w_nw_n^*$ and $q_n = w_n^*w_n$. Then
  $$N(w_n) \le N(p_n) \le N \Big( p- (\sum_{i=1}^{n-1} p_i) \Big)  < 2^{-n+1} \, ,$$
  and it follows that the sequence $\{\sum_{i=1}^n w_i\}_n$ converges to a partial isometry $w\in pRq$ such that $p=ww^*$ and $q=w^*w$.
  Hence $R$ satisfies condition $\LPeqRP$ (by \cite[Lemma 1.1]{A87}).
   \end{proof}

 In order to state the local condition in our main result of this section, we need the following somewhat technical definition.

 \begin{definition}
  \label{def:hereditarily-quasistandard}
  Let $R$ be a unital $*$-regular ring with pseudo-rank function $N$, and let $\mathcal A$ be a unital $*$-subalgebra which is $*$-isomorphic to a standard matricial $*$-algebra.
  We say that a projection $p\in \mathcal A$ is {\it hereditarily quasi-standard} if
  \begin{enumerate}
   \item $p$ is $*$-equivalent in $\mathcal A$ to a standard projection of $\mathcal A$, and,
   \item  for each subprojection $p'\le p$, $p'\in \calA $, and each $\varepsilon >0$ there exists
  a unital $*$-subalgebra $\mathcal A'$ of $R$ and a projection $p''\in \mathcal A'$ satisfying the following properties:
  \begin{enumerate}[(a)]
   \item $\mathcal A'$ is $*$-isomorphic to a standard matricial $*$-algebra,
   \item $p''$ is $*$-equivalent in $\mathcal A'$ to a standard projection of $\mathcal A'$,
   \item $p''\le p'$ and $N(p'-p'')<\varepsilon$, and
   \item $\mathcal A \subseteq \mathcal A'$.
     \end{enumerate}
    \end{enumerate}
    \end{definition}

We can now state the following analogue of Theorem \ref{thm:local-ultra}. By a continuous $*$-factor over $F$ we mean a $*$-regular ring $\calQ$ which is a $*$-algebra over $F$, and which is a
continuous factor in the sense of Section \ref{sec:generalvN}.

\begin{theorem}
 \label{thm:local-ultra2} Let $(F,*)$ be a field with positive definite involution.
Let $\calQ$ be a continuous $*$-factor over $F$, and assume that there exists a dense subalgebra (with respect to the $N_{\calQ}$-metric topology) $\calQ_0 \subseteq \calQ$
of countable $F$-dimension. The following are equivalent:
\begin{enumerate}[(i)]
 \item $\calQ \cong \calM_F$ as $*$-algebras.
 \item $\calQ$ is isomorphic as a $*$-algebra to $\overline{\calB}$ for a certain standard ultramatricial $*$-algebra $\calB$, where the completion of $\calB$ is taken with respect to the metric induced by
 an extremal pseudo-rank function on $\calB$.
 \item  For every $\varepsilon > 0$, elements $x_1,...,x_n \in \calQ$, and projections $p_1,p_2\in \calQ$,  there exist a $*$-subalgebra $\calA$ of $\calQ$,
 which is $*$-isomorphic to a standard matricial $*$-algebra, elements $y_1,...,y_n \in \calA$, and hereditarily quasi-standard projections $q_1, q_2\in \calA$
 such that
$$ N(p_j-q_j) <\varepsilon, \qquad (j=1,2), \qquad \text{and} \qquad  N(x_i-y_i) < \varepsilon \qquad (i=1,\dots , n ) .$$
\end{enumerate}
\end{theorem}

\begin{proof}
Clearly, (i)$\implies$ (ii).

\noindent (ii)$\implies $(iii).
Write $\calB =\varinjlim_n (\calB_n, \Phi_n)$ as a direct limit of a sequence of standard matricial $*$-algebras $\calB_n$ and
standard maps $\Phi_n \colon \calB_n\to \calB_{n+1}$. Write $\Phi_{ji}\colon \calB_i \to \calB_j$ for the composition maps
$\Phi_{j-1}\circ \cdots \circ \Phi_i$, for $i<j$, and write $\theta _i\colon \calB_i\to \ol{\calB}\cong \calQ$ for the canonical map. We identify
$\calQ$ with $\ol{\calB}$.

We will show that the desired $*$-subalgebra $\mathcal A$ satisfying the required conditions is of the form $\theta_j (\mathcal B_j)$. Since those algebras form an
increasing sequence, whose union is dense in the $N_{\calQ}$-metric topology, we see that it is enough
to deal only with the projections, and indeed that it is enough to deal with a single projection.
Let $p$ be a projection in $\calQ$ and let $\varepsilon > 0$.

Now there is some $i\ge 1$ and an element $x\in \calB _i$ such that $N(p-\theta _i (x)) < \varepsilon /8$.  Write $p_1:= \mathrm{LP}(x)\in \calB_i$.
By Lemma \ref{lem:prpr-N_Star}(c), we have
$$N(p - \theta _i(p_1)) = N(\mathrm{LP}(p) - \theta_i (\mathrm{LP}(x))) =  N(\mathrm{LP}(p) - \mathrm{LP}(\theta_i(x))) \le  4 N(p-\theta_i (x)) < \frac{\varepsilon}{2} ,$$
so that $N(p - \theta_i (p_1))<  \varepsilon/2$.

There exists a standard projection $g$ in $\calB_i$ such that $p_1 \sim g$ in $\calB_i$.
 By the proof of \cite[Theorem 3.5]{A87}, there are $j>i$ and projections $p_1', g'\in \calB_j$ such that
 $p_1' \le \Phi_{ji}(p_1)$, $g' \le \Phi_{ji}(g)$, $g'$ is a standard projection, $p_1' \overset{*}{\sim} g'$,
 and moreover $N(\theta_i (p_1) - \theta _j (p_1'))<\varepsilon / 2$ and $N(\theta_i(g) -\theta_j (g')) < \varepsilon /2$.
 Therefore, $p_1'$ is $*$-equivalent to a standard projection in $\calB_j$, and moreover
 $$N (p - \theta_j (p_1')) \le N( p - \theta _i(p_1)) + N( \theta_i(p_1) - \theta _j(p_1')) < \varepsilon /2 +\varepsilon / 2= \varepsilon.$$
 Now take $\calA := \theta_j(\calB _j)$ and $q:= \theta _j (p_1')$. Clearly, property (1) in Definition \ref{def:hereditarily-quasistandard} is satisfied.
 To show property (2), take a subprojection $\theta _j(p')$ of $q=\theta _j (p_1')$, where $p'$ is a subprojection of $p_1'$, and $\delta >0$. Then we use
 the same argument as above but now applied to the projection $p'$ of $\mathcal B_j$ and to $\delta >0$. We obtain $k\ge j$ and a projection $\theta _k(p'')$ in the $*$-subalgebra
 $\theta _k(\calB _k)$ such that the pair $(\theta _k(p'), \theta_k (\calB _k) )$ satisfies properties (a)-(d) in Definition  \ref{def:hereditarily-quasistandard}
 (with $\varepsilon$ replaced with $\delta$).

\noindent (iii)$\implies $(i). We first show that $\calQ$ satisfies $\LPeqRP$. Let $p_1, p_2$ be equivalent projections of $\calQ$ and $\varepsilon >0$. Choose  $x\in p_1\calQ p_2$
 and $y\in p_2\calQ p_1$ such that $p_1=xy$ and $p_2= yx$. Observe that necessarily $y=\ol{x}$, the relative inverse of $x$
 in $\calQ$.
 By (iii), there
 exists a $*$-subalgebra $\calA$ of $\calQ$, which is $*$-isomorphic to a standard matricial $*$-algebra, projections $q_1,q_2\in \calA$ such that $N(p_i-q_i) <\varepsilon$, and $q_i\overset{*}{\sim} e_i$
 in $\calA$, $i=1,2$,  for some standard projections $e_1,e_2\in\calA$, and an element $x_1\in \calA$ such that $N(x-x_1) < \varepsilon $.    Now set $x_1' := q_1x_1q_2\in \calA$, and note that
 $$N(x-x_1') \le N(p_1xp_2- q_1xq_2) + N(q_1xq_2-q_1x_1q_2) <  2\varepsilon + \varepsilon = 3 \varepsilon. $$
 It follows from Lemma \ref{lem:prpr-N_Star}(c) that, with $q_1'':= {\rm LP}(x_1')\in \calA$ and $q_2'':= {\rm RP}(x_1')\in \calA$, we have
 $$ N(p_i- q_i'') < 12 \varepsilon , \qquad q_i''\le q_i , \qquad (i=1,2) .$$
 Moreover, we have $q_1'' = {\rm LP}(x_1') \sim {\rm RP}(x_1') = q_2''$. In addition, we get
 $$N(q_i-q_i'')\le N(q_i-p_i) + N(p_i-q_i'') < \varepsilon + 12 \varepsilon = 13 \varepsilon.$$
 Write $\eta = 13 \varepsilon$. Since $q_i \overset{*}{\sim} e_i$ in $\calA$, we obtain in particular projections $e_i''\le e_i$
 such that $e_1''\sim e_2''$ (in $\calA$) and $N(e_i-e_i'')< \eta $. Now $\calA$ is a standard matricial $*$-algebra, and the restriction of $N$ to $\calA$ is a convex
 combination of the normalized rank functions on the different simple components of $\calA$, so the above information enables us to build standard projections $e_i'\le e_i$
 such that $e_1'\overset{*}{\sim} e_2'$, and $N(e_i - e_i') < \eta$ for $i=1,2$. This in turn gives us projections $q_i'\le q_i$ (through the $*$-equivalences $q_i \overset{*}{\sim} e_i$)
 such that  $q_1' \overset{*}{\sim} q_2'$ and $N(q_i-q_i') < \eta$ for $i=1,2$.

 The last step is to transfer these to $p_1, p_2$. For this, observe that
 $$N(p_i-q_i') \le N(p_i-q_i) + N(q_i-q_i') < \varepsilon + \eta .$$
 Since moreover $q_1'$ and $q_2'$ are $*$-equivalent, it follows from Lemma \ref{lem:prpr-N_Star}(d) that there exist projections $p_i'\le p_i$
 such that $p_1'\overset{*}{\sim} p_2'$ and $N(p_i-p_i')< 5(\varepsilon + \eta ) = 70 \varepsilon$. Now we can apply Lemma \ref{lem:LPRPapproximated} to conclude that $\calQ$ satisfies
 $\LPeqRP$.

Now (i) is shown by using the same method employed in Section \ref{sec:generalvN}.
We only need to prove a variant of Lemma \ref{lem:sequences-and-morphisms-local} with $*$-algebra homomorphisms $\rho_i \colon  M_{p_i}(F)\to \calQ$
instead of just algebra homomorphisms. For this, a new version of Lemmas \ref{lem:prel-lemmaapprox} and \ref{lem:technical3} is required, as follows:

 \begin{lemma}
  \label{lem:approx-projections}
 Let $p$ be a positive integer. Then there exists a constant $K^*(p)$, depending only on $p$, such that for any field with involution $F$, for any $\varepsilon >0$, for any pair $\mathcal A \subseteq \mathcal B$,
 where $\mathcal B$ is a unital $*$-algebra over $F$, and
 $\mathcal A$ is a unital
$*$-regular subalgebra of $\mathcal B$, for any pseudo-rank function $N$ on $\mathcal B$ such that $N(b^*)= N(b)$ for all $b\in \calB$, and for every $*$-algebra homomorphism $\rho \colon M_p(F) \to \mathcal B$ such that
 $\{ \rho (e_{ij}) \mid i,j =1, \dots , p \} \subseteq_{\varepsilon}  \mathcal A  $ with respect to the $N$-metric, where
$e_{ij}$ denote the canonical matrix units in $M_p(F)$, there exists
a $*$-algebra homomorphism $\psi \colon M_p(F) \to \mathcal A$ such
that
 $$N(\rho (e_{ij}) - \psi (e_{ij})) < K^*(p) \varepsilon \quad \text{ for } 1\le ,i,j\le p .$$
 If, in addition, we are given a projection $f\in \calA$ such that $N(\rho (e_{11}) -f)<\varepsilon$, then the map $\psi$ can be built with the additional property that
 $\psi (e_{11})\le f$.   \end{lemma}

   \noindent
{\it Proof of Lemma  \ref{lem:approx-projections}.}
 The proof follows the same steps as the proof of Lemma \ref{lem:prel-lemmaapprox}. There is only an additional degree of approximation due to the fact that we need projections instead of idempotents.
 Proceeding by induction on $p$, just as in the proof of Lemma \ref{lem:prel-lemmaapprox}, we start with $*$-matrix units $\{ x_{ij} \}$ for $1\le i,j\le p-1$, so that $x_{ji}= x_{ij}^*$ for all $i,j$, and we have
 to define new elements $y_{1i}$, for $i=1,\dots ,p$, so that the family $y_{ij}= y_{1i}^*y_{1j}$, $1\le i,j\le p$, is the desired new family of $*$-matrix units.
 To this end, one only needs to replace the idempotent $g$ found in that proof by the projection ${\rm LP}(g)$. Using Lemma \ref{lem:prpr-N_Star}, one can easily control the corresponding ranks.

  The last part is proven by the same kind of induction, starting with $\psi (1) = f$ for the case $p=1$. \qed

\begin{lemma}
 \label{lem:technical4} Assume that $\calQ$ satisfies condition (iii) in Theorem \ref{thm:local-ultra2}. Let $\theta $ be a real number such that $0< \theta <1$ and let $\{x_n\}_n$ be a $K$-basis of $\calQ_0$.
 Let $p$ be a positive integer such that there exist a $*$-algebra homomorphism $\rho \colon M_p(F) \to \mathcal \calQ$, a $*$-subalgebra
$\calA\subseteq \calQ$, which is $*$-isomorphic to a standard matricial $*$-algebra,
a hereditarily quasi-standard projection $g \in \calA$, a positive integer $m$, and $\varepsilon > 0$
 such that
\begin{enumerate}[(a)]
\item $N_{\calQ}(\rho(1)) = \frac{p}{q} > \theta$ for some positive integer $q$.
\item $N_{\calQ}(\rho(e_{11}) -g) <\varepsilon$, where $e_{ij}$ are the canonical matrix units of $M_p(F)$.
\item $\{ \rho(e_{ij}) \mid i,j=1,\dots , p \} \subseteq_{\varepsilon} \calA$, and  $\mathrm{span}\{x_1,...,x_m\} \subseteq_{\varepsilon} \calA$.
\item $\varepsilon < \frac{1}{48K^*(p)p^2} \Big( \frac{p}{q} - \theta \Big)$.
\end{enumerate}
 Then there exist positive integers $p', t, q'$, and a real number $\varepsilon' > 0$ with
 $p'=tp$, a $*$-algebra homomorphism $\rho' \colon M_{p'}(F) \to \calQ$, a $*$-subalgebra $\calA' \subseteq \calQ$, which is $*$-isomorphic
to a standard matricial $*$-algebra,
 and a hereditarily quasi-standard projection $g'\in \calA '$, such that the following conditions hold:
 \begin{enumerate}
  \item $N_{\calQ} (\rho' (1)) = p'/q' $.
  \item $$ 0 < \frac{p'}{q'} -\theta < \frac{1}{2} \Big( \frac{p}{q} -\theta \Big) .$$
  \item For each $x\in \rho (1) \calA \rho(1)$ there exists $y\in M_{p'}(F)$ such that
  $$N_{\calQ}(x - \rho ' (y)) < \frac{p}{q } - \theta .$$
  \item For each $z\in M_p (F)$, we have
  $$N_{\calQ } (\rho (z) - \rho'(\gamma (z)))< \frac{p}{q} -\theta ,$$
 where
$ \gamma \colon M_p(F) \to M_{p'}(F) = M_p(F)\otimes M_t(F)$
is the canonical unital $*$-homomorphism sending $z$ to $z\otimes 1_t$.
\item $N_{\calQ} ( \rho'(e'_{11}) - g') <\varepsilon'$, where $e'_{ij}$ are the canonical matrix units of $M_{p'}(F)$.
\item $\{ \rho '(e'_{ij}) \mid i,j=1,\dots , p' \} \subseteq_{\varepsilon'} \calA'$, and $\mathrm{span}\{x_1,...,x_m,x_{m+1}\} \subseteq_{\varepsilon'} \calA'$
\item $\varepsilon' < \frac{1}{48K^*(p'){p'}^2} \Big( \frac{p'}{q'} - \theta \Big)$.
   \end{enumerate}
 \end{lemma}

 \medskip

\noindent
{\it Proof of Lemma  \ref{lem:technical4}.}
 The proof is very similar to the proof of Lemma \ref{lem:technical3}. We only indicate the points where the proof has to be modified.

 We denote by $e_{ij}$, for $1\le i,j \le p$, the canonical matrix units in $M_p(F)$. Note that $e_{ij}^* = e_{ji}$ for all $i,j$. Set $f':= \rho (e_{11})$, which is a projection
 in $\calQ$ with $N_{\calQ}(f')= 1/q $. By hypothesis, there is a hereditarily quasi-standard projection $g$ in the $*$-subalgebra $\calA$ such that
 $N_{\calQ}(f' -g) < \varepsilon $. Now by Lemma \ref{lem:approx-projections} there exists a $*$-algebra homomorphism $\psi \colon M_p(F) \to \calA $ such that
 $\psi (e_{11})\le g$ and $N_{\calQ}(\rho (e_{ij}) - \psi (e_{ij})) < K^*(p) \varepsilon $ for all $i,j$. Now by condition (2) in Definition \ref{def:hereditarily-quasistandard}, there exists
 another $*$-subalgebra $\calA'$ of
 $\calQ$, which is $*$-isomorphic to a standard matricial $*$-algebra and contains $\calA$,
 and a projection $f\in \calA'$, which is $*$-equivalent in $\calA'$ to a standard projection of $\calA'$, such that $f\le \psi (e_{11})$
 and $N_{\calQ} (\psi (e_{11})- f) < K^*(p) \varepsilon - \mu$, where
 $$\mu = \text{max} \{ N_{\calQ} (\rho (e_{ij}) - \psi (e_{ij})):i,j=1,\dots , p \}.$$
 Now, setting $\psi ' (e_{ij})= \psi (e_{i1})f\psi (e_{1j})$, we obtain that $\psi'$ is a $*$-algebra homomorphism from $M_p(F)$ to $\calA '$, and that
 $$ N_{\calQ} (\rho (e_{ij})- \psi '(e_{ij})) < \mu + (K^*(p) \varepsilon - \mu) = K^*(p) \varepsilon ,$$
 so that, after changing notation, we may assume that $f = \psi (e_{11})$, and that $f$ is $*$-equivalent in $\calA$ to a standard projection of $\calA$.

Since $\calA$ is a standard matricial $*$-algebra, we can write $f=f_1+\cdots + f_k$, where $f_1,f_2,\dots ,f_k$ are nonzero mutually orthogonal
projections belonging to different simple factors of $\calA$. Since $f$ is $*$-equivalent in $\calA$ to a standard projection,  there exists, for each $1\le i \le k$, a set of matrix
units $\{ f^{(i)}_{jl} : 1\le j,l\le r_i \}$ inside $f_i\calA f_i$ such that each $f^{(i)}_{jj}$ is a minimal projection in the simple factor to which $f_i$ belongs, such that
$$\sum_{j=1}^{r_i} f^{(i)} _{jj}= f_i $$
for $i=1,\dots ,k$ and moreover $(f^{(i)}_{jl})^* = f^{(i)}_{lj}$ for all $i,j,l$.

Now the proof follows the same steps as the one of Lemma \ref{lem:technical3}. The idempotent $e$ built in that proof can be replaced now by a projection and, since $\calQ$
satisfies $\LPeqRP$, we have that $p_i'\cdot e $ is $*$-equivalent to a subprojection of $f^{(i)}_{11}$. Using this and the fact that $(f^{(i)}_{jl})^* = f^{(i)}_{lj}$ for all $i,j,l$, one builds
a system of matrix units inside $f\calQ f$
 $$ \{ h^{(i_1,i_2)}_{( j_1,j_2), (u_1,u_2) } : 1\le i_1,i_2\le k,\,\,  1\le j_1 \le r_{i_1}, \, \, 1\le j_2\le r_{i_2} , \, \, 1\le u_1\le p'_{i_1} , \, \,  1\le u_2 \le p'_{i_2} \}, $$
  satisfying all the conditions stated in the proof of \ref{lem:technical3}, and in addition
  $$(h^{(i_1,i_2)}_{( j_1,j_2), (u_1,u_2)})^* = h^{(i_2,i_1)}_{( j_2,j_1), (u_2,u_1)}$$
  for all allowable indices.

We can now define a $*$-algebra homomorphism $\rho'\colon M_{p'}(K)= M_p(K)\otimes M_t(K)\to \mathcal \calQ$ by the rule
$$\rho ' (e_{ij}\otimes e^{(i_1,i_2)}_{(j_1,j_2),(u_1,u_2)}) = \psi (e_{i1})  h^{(i_1,i_2)}_{( j_1,j_2), (u_1,u_2) } \psi (e_{1j}) \, ,$$
where $\{ e^{(i_1,i_2)}_{(j_1,j_2),(u_1,u_2)} \}$ is a complete system of $*$-matrix units in $M_t(K)$.

The verification of properties (1)-(7) is done in the same way, using condition (iii) to show that conditions (5) and (6) are satisfied.
 \qed

 \bigskip

 Lemma \ref{lem:technical4} enables us to build the sequence of $*$-algebra homomorphisms $\rho_i \colon  M_{p_i}(F)\to \calQ$ satisfying the properties stated in
 Lemma \ref{lem:sequences-and-morphisms-local}, and the same proof gives a $*$-isomorphism from $\calM$ to $\calQ$, as desired.
 \end{proof}

In case the base field with involution $(F,*)$ is *-Pythagorean, we can derive a result which is completely analogous to Theorem \ref{thm:local-ultra}, as follows.

\begin{corollary}
 \label{cor:Pythagorean-thm}
  Let $(F,*)$ be a *-Pythagorean field with positive definite involution.
Let $\calQ$ be a continuous $*$-factor over $F$, and assume that there exists a dense subalgebra (with respect to the $N_{\calQ}$-metric topology) $\calQ_0 \subseteq \calQ$
of countable $F$-dimension. The following are equivalent:
\begin{enumerate}[(i)]
 \item $\calQ \cong \calM_F$ as $*$-algebras.
 \item $\calQ$ is isomorphic as a $*$-algebra to $\overline{\calB}$ for a certain ultramatricial $*$-algebra $\calB$, where the completion of $\calB$ is taken with respect to the metric induced by
 an extremal pseudo-rank function on $\calB$.
 \item  For every $\varepsilon > 0$ and elements $x_1,...,x_n \in \calQ$, there exist a matricial $*$-subalgebra $\calA$ of $\calQ$,
and elements $y_1,...,y_n \in \calA$ such that $$N_{\mathcal Q} (x_i-y_i) < \varepsilon \qquad (i=1,\dots , k ) .$$
\end{enumerate}
\end{corollary}

\begin{proof}
 This follows from Theorem \ref{thm:local-ultra2}, by using the fact that $M_n(F)$ satisfies $\LPeqRP$  for all $n$ (\cite[Theorem 4.9]{HandRocky})
 and \cite[Proposition 3.3]{A87}. Note that, since $M_n(F)$ satisfies $\LPeqRP$ for all $n$, every projection of a standard matricial $*$-algebra is hereditarily quasi-standard.
 \end{proof}

 \section*{Acknowledgments}

The authors would like to thank Andrei Jaikin-Zapirain and Kevin O'Meara for their useful suggestions. The authors also thank the referee for his/her very careful reading of the manuscript and
for his/her many suggestions.


\begin{thebibliography}{12}

\bibitem{A87} P. Ara, \emph{Matrix rings over *-regular rings and pseudo-rank functions},
Pacific J. Math. {\bf 129} (1987), 209--241.

\bibitem{AC} P. Ara, J. Claramunt, \emph{Approximating the group algebra of the lamplighter by finite-dimensional algebras}.
In preparation.

\bibitem{Berb} S. K. Berberian, \emph{Baer $*$-rings}. Die Grundlehren der mathematischen Wissenschaften, Band 195. Springer-Verlag, New York-Berlin, 1972.


\bibitem{BH} B. Blackadar, D. Handelman, \emph{Dimension functions and traces on C*-algebras},
J. Funct. Anal. {\bf 45} (1982), 297--340.

\bibitem{Cohn} P. M. Cohn, \emph{Algebra. Vol. 3},  Second edition, John Wiley \& Sons, Ltd., Chichester, 1991.

\bibitem{connes} A. Connes, \emph{Classification of injective factors. Cases II$_1$, II$_{\infty}$, III$_{\lambda}$, $\lambda \ne 1$}, Ann. of Math. (2) {\bf 104} (1976), 73--115.

\bibitem{elekTRANS} G. Elek, \emph{Connes embeddings and von Neumann regular closures of amenable group algebras},  Trans. Amer. Math. Soc. {\bf 365} (2013), 3019--3039.

\bibitem{elek}
  G. Elek,
  \emph{Lamplighter groups and von Neumann continuous regular rings},
Proc. Amer. Math. Soc. {\bf 144} (2016), 2871--2883.

\bibitem{elek2}
G. Elek,
\emph{Infinite dimensional representations of finite dimensional algebras and amenablility}, Math. Annalen {\bf 369} (2017), 397--439.

\bibitem{HandCMB} D. Handelman, \emph{Completions of rank rings}, Canad. Math. Bull. {\bf 20} (1977), 199--205.

\bibitem{HandTAMS} D. Handelman, \emph{Coordinatization applied to finite Baer *-rings}, Trans. Amer. Math. Soc. {\bf 235} (1978), 1--34.

\bibitem{HandRocky} D. Handelman, \emph{Rings with involution as partially ordered abelian groups}, Rocky Mountain J. Math. {\bf 11} (1981), 337--381.

\bibitem{MvN} F. J. Murray, J. von Neumann, \emph{On rings of operators. IV},  Ann. of Math. (2) {\bf 44} (1943), 716--808.

\bibitem{Good-centers} K. R. Goodearl, \emph{Centers of regular self-injective rings}, Pacific J. Math. {\bf 76} (1978), 381--395.

\bibitem{vnrr}
K. R. Goodearl, ``Von Neumann Regular Rings'',
Pitman, London 1979; Second Ed., Krieger, Malabar, Fl., 1991.

\bibitem{Halp} I. Halperin, \emph{Von Neumann's manuscript on inductive limits of regular rings}, Canad. J. Math. {\bf 20} (1968), 477--483.

\bibitem{Jaikin-survey} A. Jaikin-Zapirain, \emph{$L^2$-Betti numbers and their analogues in positive characteristic}, to appear in Proceedings of Groups St. Andrews 2017.

\bibitem{SS} A. M. Sinclair, R. R. Smith,  \emph{Finite von Neumann algebras and masas}, London Mathematical Society Lecture Note Series, 351.
Cambridge University Press, Cambridge, 2008.

\end{thebibliography}
\end{document}